\documentclass[11pt,reqno]{amsart}
\usepackage{amsmath,amsthm,amssymb,enumerate}
\usepackage{hyperref}
\usepackage{mathrsfs}

\usepackage{color}

\newtheorem{theorem}{Theorem}[section]
\newtheorem{lemma}[theorem]{Lemma}
\newtheorem{proposition}[theorem]{Proposition}
\newtheorem{corollary}[theorem]{Corollary}

\theoremstyle{definition}
\newtheorem{assumption}{Assumption}[section]
\newtheorem{definition}{Definition}[section]

\theoremstyle{remark}
\newtheorem{remark}{Remark}[section]

\newcommand\bR{\mathbb{R}}

\newcommand*{\beq}{\begin{equation}}
\newcommand*{\eeq}{\end{equation}}
\newcommand{\E}{\mathbb{E}}

\newcommand{\R}{\mathbb{R}}
\newcommand{\bit}{\begin{itemize}}
\newcommand{\eit}{\end{itemize}}
\newcommand\D{\partial}

\begin{document}

\title{On the boundedness of solutions of SPDEs
}

\author{Konstantinos Dareiotis    \and
        M\'at\'e Gerencs\'er 
}


\maketitle

\begin{abstract}
In this paper estimates for the $L_\infty-$norm of solutions of parabolic SPDEs are derived. The result is obtained through iteration techniques, motivated by the work of  Moser  in deterministic settings. As an application of the main result, solvability of a class of semilinear SPDEs is established.
\end{abstract}

\section{Introduction}         \label{introduction}
In the present work we consider the following stochastic partial differential equation (SPDE) on $[0,T] \times Q$,
\begin{equation}                                          \label{SPDE}
du_t=(L_tu_t+\D_if^i_t+f^0_t)dt+(M^k_tu_t+g^k_t)dw^k_t, \ u_0 =\psi,
\end{equation}
where the operators $L_t$, and $M^k_t$ are given by 
$$
L_tu=\D_j(a^{ij}_t\D_iu)+b^i_t\D_iu+c_tu, \ M^k_tu=\sigma^{ik}_t\D_iu_t+\mu^k_tu,
$$
with merely bounded and measurable coefficients, and $Q$ is a bounded Lipschitz domain in $\bR^d$. We use the summation convention with respect to integer valued repeated indices. In particular, the summation for the parameters $i$ and $j$ takes place over the set $\{1,...,d\}$,  and for $k$ over the positive integers. We are interested in boundedness properties of weak solutions under a strong stochastic parabolicity condition. The corresponding problem in the deterministic case, has been extensively studied. The first results for non-degenerate equations in divergence form are due to  \cite{DG} and \cite{MOS} for the elliptic case and \cite{NASH} for both elliptic and parabolic equations. Later, the techniques of  \cite{MOS} were extended to the parabolic case in  \cite{MOS1}. The approach of \cite{DG} was also applied for parabolic equations (see for example \cite{LA}). In fact, in all these articles, not only boundedness, but stronger results are obtained, namely,  H\"older continuity and Harnack inequalities. Another proof of the parabolic Harnack inequality was given in \cite{FAB}.  H\"older estimates and Harnack inequality were also obtained  in \cite{SAF} and \cite{KRSAF}, for elliptic and parabolic equations in non-divergence form. More recently, these results were also proved for a wider class of parabolic equations, including, for example, the $p$-Laplacian as the driving operator (see \cite{DIBE} and references therein).

Boundedness of solutions of SPDEs can be proved through embedding theorems of Sobolev spaces. Such results can be obtained from $L_p-$theory, see e.g. \cite{KrylovLP}, for equations considered on the whole space. This approach, however, requires some regularity of the coefficients. For SPDEs where these regularity assumptions are dropped or weakened, the literature has been expanding recently. In \cite{BSPDE} a maximum principle is obtained for a class of backward SPDEs. Under the additional assumption $\sigma=0$, variants of the problem are treated in \cite{DDH}, \cite{HWW}, and \cite{Kim}, with methods that strongly rely on the absence of derivatives of $u$ in the noise term.  In \cite{DMS}, through the technique of Moser's iteration, introduced in \cite{MOS}, boundedness results are derived without posing regularity  assumptions on the coefficients, for a class of quasilinear equations, by staying in the $L_2-$framework. This served as a main motivation to our work.  However, in \cite{DMS}, it is assumed that there exist constants $\lambda> \beta>0$, such that for any $\xi \in \mathbb{R}^d$, one has $a^{ij}\xi_i\xi_j \geq \lambda|\xi|^2$ and $(72+1/2)\sigma^{ik}\sigma^{jk} \xi_i\xi_j \leq \beta|\xi|^2$.  Consequently, the important case of linear SPDEs appearing in filtering theory is not covered. In the present paper only the classical stochastic parabolicity condition will be assumed in order to get estimates for the uniform bound of the solution of equation \eqref{SPDE}. We note that the results of the present paper can also be extended to quasilinear equations under suitable conditions. Having accessibility in mind, in the present work such generalizations are not included.

With the use of our main theorem, existence and uniqueness results for semilinear SPDEs are derived, under a weak condition on the growth of the semi-linear term $f(u)$ (see equation \eqref{SLSPDE} in section \ref{section-semilinear}). We construct the solutions by using comparison techniques, adopted from \cite{GP}.

Let us introduce some of the notation that will be used through the paper (for general notions on SPDEs we refer to \cite{ROZ} and \cite{ROK}). We consider  
a complete probability space $(\Omega, \mathscr{F}, P)$. It is equipped with a 
right-continuous filtration $(\mathscr{F}_t)_{t\geq0}$, 
such that $\mathscr{F}_0$ contains all 
$P$-zero sets, and $\{w^k_t\}_{k=1}^{\infty}$ is a sequence of 
independent 
real valued  $\mathscr{F}_t$-Wiener processes on $\Omega$. 
The set of all compactly supported smooth functions on $Q$, will be denoted by $C^\infty_c(Q)$. We will use the notation  $H^1_0(Q)$ the space of all measurable functions $v$ on $Q$, vanishing on the boundary, such that $v$ and its generalized derivatives of first order lie in $L_2(Q)$. The inner product in $L_2(Q)$, will be denoted by $(\cdot,\cdot)$. 
For $p,r,q\in [1,\infty]$, the norm in $L_p(Q)$ will be denoted by $|\cdot |_p$, while the norm in $L_{r,q}:=L_r([0,T];L_q( Q))$ will be denoted by $\|\cdot\|_{r,q}$. If $q=r$, then for simplicity we will write $\|\cdot \|_r$ instead of $\|\cdot\|_{r,r}$. 
We set $\mathbf{L}_p:=L_p(\Omega, \mathscr{F}_0;L_p(Q))$, $\mathbb{L}_p:=L_p(\Omega\times [0,T], \mathscr{P} ;L_p(Q))$, and $\mathbb{L}_p(l_2):=L_p(\Omega\times [0,T], \mathscr{P} ;L_p(Q;l_2))$ where $\mathscr{P}$ is the predictable $\sigma$-algebra. The constants in the calculations, usually denoted by $N$, may change from line to line, but, unless otherwise noted, they always depend only on the structure constants of the equation (see Section \ref{section-mainresult}). 

The rest of the paper is organized as follows. In Section \ref{section-mainresult} the assumptions are formulated and the main theorems are stated. In Section \ref{section-preliminaries} preliminary results are collected, which are then used in the proof of the main theorem in Section \ref{section-proof}. In Section \ref{section-semilinear}, we apply our result, in combination with a comparison principle, to construct solutions for a class of semilinear SPDEs.
\section{Formulation and Main Results}\label{section-mainresult}

We pose the following conditions on equation \eqref{SPDE}. 
\begin{assumption}                    \label{as:coef and free terms}
i)  The coefficients $a^{ij}$, $b^i$ and $c$ are real-valued 
$\mathscr{P} \times \mathscr{B}(Q)$ measurable functions 
on $\Omega\times[0,T]\times Q$ and are  
bounded by  a constant $K\geq0$, for any $i,j=1,...,d$. 
The coefficients $\sigma^i=(\sigma^{ik})_{k=1}^{\infty}$ and $\mu=(\mu^k)_{k=1}^\infty $ 
are  $l_2$-valued 
$\mathscr{P} \times Q$-measurable functions on  $\Omega\times[0,T]\times Q$ such that 
$$
\sum_i \sum_k |\sigma^{ik}_t(x)|^2 +\sum_k|\mu^k_t(x)|^2\leq K \quad 
\text{for all $\omega$, $t$ and $x$},
$$
ii)$f^l$, for $l \in \{0,...,d\}$,  and $g=(g^k)_{k=1}^\infty$ are $\mathscr{P} \times \mathscr{B}(Q)$-measurable functions on $\Omega \times [0,T] \times Q $ with values in $\bR$ and $l_2$, respectively, 
such that 
$$
\E(\sum_{l=0}^d \|f^l \|^2_2+\||g|_{l_2}\|^2_2) < \infty
$$
iii) $\psi$ is an $\mathscr{F}_0$-measurable random variable in $L_2(Q)$ 
such that $\E |\psi|^2_2< \infty$
\end{assumption}
\begin{assumption}[Parabolicity]                          \label{as: parabolicity}
There exists a constant $\lambda >0$ such that for all $\omega,t,x$ and for all $ \xi =( \xi _1,...\xi _d) \in \R^d$ we have
\[ a^{ij}_t(x)\xi_i \xi_j  -\frac{1}{2}  \sigma^{ik}_t(x)\sigma^{jk}_t(x)\xi_i \xi _j 
\geq \lambda |\xi |^2,\]
\end{assumption}
We will refer to the constants $K,T,\lambda, d$ and $|Q|$, where the latter is the Lebesgue measure of $Q$, as structure constants. 
\begin{definition}
An $L_2-$solution of equation \eqref{SPDE} is understood to be an $L_2(Q)$-valued, $ \mathscr{F}_t-$adapted, strongly continuous process $(u_t)_{t\in [0,T]}$, such that 

i) $u_t \in H^1_0(Q), \ \text{for $dP \times dt$ almost every} \  (\omega,t) \  \in \Omega\times [0,T]$ 

ii) $  \E\int_0^T( |u_t|_2^2+|\nabla u_t|_2^2 )dt< \infty $

iii) for all $\phi\in C_c^\infty(Q)$ we have with probability one
\[(u_t,\phi)=(\psi,\phi)+\int_0^t -(a^{ij}_s\D_iu_s+f^j_s ,\D_j \phi)+(b^i_s\D_iu_s+c_su_s+f^0_s, \phi)ds \]
\[+\int_0^t(M^k_su_s+g^k_s,\phi)dw^k_s, \]
 for all $t \in [0,T]$.
 
\end{definition}
Equation \eqref{SPDE} can be understood as a stochastic evolution equation on the Gel'fand  triple $H^1_0(Q) \hookrightarrow L_2(Q)\hookrightarrow H^{-1}(Q)$.   Assumptions \ref{as:coef and free terms} and \ref{as: parabolicity}   ensure that the standard conditions for solvability of this type of equations (conditions $A_1$-$A_5$ from \cite{KSEE}) are satisfied. Therefore, by Theorems 3.6 and 3.10 from \cite{KSEE}, equation \eqref{SPDE} admits a unique $L_2-$solution $u$, and the following estimate holds
\begin{equation}\label{L2 supremum estimate}
\E\sup_{0\leq t\leq T}|u_t|_2^2\leq N\E(|\psi|_2^2+\sum_{l=0}^d\|f^l\|_2^2+\||g|_{l_2}\|_2^2),
\end{equation}
where $N=N(d,K,\lambda,T)$.

Let
$$\Gamma_d=\left\lbrace(r,q)\in (1,\infty]^2\middle|\frac{1}{r}+\frac{d}{2q}<1\right\rbrace.$$
The following is our main result.
\begin{theorem}                                \label{thm:main theorem}
Suppose that Assumptions \ref{as:coef and free terms} and \ref{as: parabolicity} hold, and let $u$ be the unique $L_2-$ solution of equation \eqref{SPDE}. Then for any $(r, q)\in \Gamma_d$ and $\eta>0$, 
\begin{equation}                                            \label{eq:main estimate}
\E\|u\|_{\infty}^\eta \leq N \E(|\psi|_{\infty}^\eta+\|f^0\|_{r,q}^\eta+\sum_{i=1}^d\|f^i\|_{2r,2q}^\eta+\||g|_{l_2}\|_{2r,2q}^\eta),
\end{equation}     
where $N=N(\eta,r,q,d,K, \lambda,|Q|,T)$.
\end{theorem}
\begin{remark}
Notice that in particular we obtain
\begin{equation}                                            \label{eq: estimate infty}
\E\|u\|_{\infty}^2 \leq N \E(|\psi|_{\infty}^2+\sum_{l=0}^d\|f^l\|_\infty^2+\||g|_{l_2}\|_\infty^2),
\end{equation} 
and by interpolating between \eqref{L2 supremum estimate} and \eqref{eq: estimate infty}, for any $p\geq2$, one obtains
$$\E\sup_{0\leq t\leq T}|u_t|_p^2\leq N\E(|\psi|_p^2
+\sum_{l=0}^d\|f^l\|_p^2+\||g|_{l_2}\|_p^2)
$$
where $N$ can be chosen to be independent of $p$. In fact, such a uniform estimate for the $L_p$-norms of the solutions is equivalent to \eqref{eq: estimate infty}.
\end{remark}
Theorem \ref{thm:main theorem} will be proved in Section \ref{section-proof}. We will adapt the technique of Moser from \cite{MOS} and  \cite{MOS1}.  The strategy, in short, and for the moment ignoring the contributions from the initial and free data, is the following: with a suitable intermediate norm $[u]_n$ we obtain estimates of the form
$\E\|u\|_{r_{n+1},q_{n+1}}^\eta \leq N(n)\E[u]_n^{\eta}$, $\E[u]_n^{\eta}\leq N(n) \E\|u\|^\eta_{r_n,q_n}$, with $r_n,q_n\nearrow\infty$. The constants $N(n)$ in these estimates are controlled so that one can iterate this procedure, take limits, and finally obtain estimates for the supremum norm.

\section{Preliminaries}\label{section-preliminaries}

In this section we gather some results  that we will need for the proof of Theorem \ref{thm:main theorem}. First let us invoke (II.3.4) from \cite{LA}.
\begin{lemma}\label{embedding}
Suppose that $v\in L_2([0,T],H^1_0(Q))\cap L_{\infty}([0,T], L_2(Q))$. Let $r,q\in(2,\infty)$, satisfying $1/r+d/2q=d/4$. Then $v$ belongs to $ L_{r}([0,T],L_q(Q))$, and
$$
\left(\int_0^T\left(\int_Q|v_t|^{q}dx\right)^{r/q}dt\right)^{2/r}\leq N\left(\sup_{0\leq t\leq T}\int_Q|v_t|^2dx+\int_0^T\int_Q |\nabla v_t|^2dxdt\right)
$$
with $N=N(d,|Q|,T)$.
\end{lemma}
The right hand side of the inequality in the above lemma plays the role of the ``suitable norm" (for $n=2$), which was discussed  at the end of the previous section.
We are also going to use the following result (see Proposition IV.4.7 and Exercise IV.4.31/1, \cite{MY}).
\begin{proposition}\label{Revuz-Yor}
Let $X$ be a non-negative, adapted, right-continuous process, and let $A$ be a non-decreasing, continuous process such that
$$\E (X_{\tau}| \mathscr{F}_0 )\leq \E (A_{\tau}| \mathscr{F}_0)$$
for any bounded stopping time $\tau$. Then for any $\sigma\in(0,1)$
$$\E \sup_{t\leq T}X_t^{\sigma}\leq \sigma^{-\sigma}(1-\sigma)^{-1}\E A_T^{\sigma}.$$
\end{proposition}

In order to obtain our estimates, we will need and It\^o formula for $|u_t|_p^p$. The difference between the next lemma and Lemma 8 in \cite{DMS}, is that we obtain supremum (in time) estimates, that are essential  for having \eqref{eq:Ito formula for |u|^p} almost surely, for all $t \in [0,T]$. Therefore, we give a whole proof for the sake of completeness.
\begin{lemma}                                    \label{lem: moment bounds for bounded f}

Suppose that $u$ satisfies equation  \eqref{SPDE},  $f^l \in \mathbb{L}_p$, for $l\in \{0,...,d\}$,  $\  g \in \mathbb{L}_p(l_2)$, and $\psi \in \mathbf{L}_p$ for some $p\geq2$. Then there exists a constant $N=N(d,K,\lambda,p)$, such that

\begin{equation}                                 \label{eq: finiteness of |u|^p}
\E \sup_{t \leq T} |u_t|^p_p + \E\int_0^T \int_Q |\nabla u_s|^2|u_s|^{p-2} dxds \leq N \E(|\psi|_p^p+\sum_{l=0}^d \|f^l\|_p^p+\| |g|_{l_2}\|^p_p).
\end{equation}

Moreover, almost surely

\begin{align}
\int_Q |u_t|^pdx &=\int_Q |u_0|^pdx + p\int_0^t\int_Q (\sigma^{ik}_s \D_i u_s +\mu^ku_s+g^k) u_s|u_s|^{p-2}   dx dw^k_s \nonumber\\
&+ \int_0^t \int_Q -p(p-1)a^{ij}_s\D_iu_s |u_s|^{p-2}\D_j u_s -p(p-1)f^i_s\D_iu_s  |u_s|^{p-2} dxds\nonumber \\
&+\int_0^t \int_Q p(b^i_s\D_i u_s + c_su_s  +f^0_s ) u_s|u_s|^{p-2}dxds \nonumber\\
&+\frac{1}{2}p(p-1)\int_0^t\int_Q \sum_{k=1}^\infty |\sigma^{ik}_s \D_iu_s+ \mu^ku_s+g^k_s|^2 |u_s|^{p-2}dxds, \label{eq:Ito formula for |u|^p}           
\end{align}

for any $t \leq T$. 

\end{lemma}

\begin{proof}

Consider the functions 

\begin{equation}
 \phi_n (r)= \left\{
\begin{array}{rl}
|r|^p \ \ \ \ \ \ \ \ \ \ \ \ \ \  \ \ \qquad \qquad \qquad\qquad \qquad  & \text{if } |r| < n \\
n^{p-2}\frac{p(p-1)}{2}(|r|-n)^2+pn^{p-1}(|r|-n)+n^p & \text{if } |r| \geq n. 
\end{array} \right. \nonumber
\end{equation}
Then one can see that $\phi_n$  are twice continuously differentiable, and satisfy

$$
|\phi_n(x)| \leq N|x|^2, \ |\phi_n'(x)| \leq N|x|, \ |\phi''_n(x)| \leq N,
$$
where $N$ depends only on $p$ and $n\in \mathbb{N}$. We also have that for any $r  \in \bR$, $\phi_n(r)\to |r|^p, \ \phi_n'(r) \to p|r|^{p-2}r$, $ \phi_n''(r)\to p(p-1)|r|^{p-2}$, as $n \to \infty$, and 

\begin{equation}                                   \label{eq:dom of phi n}
\phi_n(r)\leq N |r|^p, \ \phi_n'(r) \leq N |r|^{p-1}, \  \phi_n''(r)\leq N |r|^{p-2}, 
\end{equation}
 where $N$ depends only on $p$. Then for each $n \in \mathbb{N}$ we have  almost surely

\begin{align}
\int_Q \phi_n(u_{t})dx &=\int_Q \phi_n(u_0) dx+ \int_0^{t}\int_Q (\sigma^{ik}_s \D_i u_s +\mu^ku_s+g^k) \phi_n'(u_s)   dxdw^k_s\nonumber\\
&+ \int_0^{t} \int_Q -a^{ij}_s\D_iu_s \phi''_n(u_s)\D_j u_s -f^i\phi''_n(u_s)\D_i u_s dxds \nonumber \\
&+\int_0^{ t}\int_Q  b^i_s\D_i u_s \phi'_n(u_s) + c_su_s \phi_n'(u_s) +f^0_s \phi_n'(u_s) dxds \nonumber\\                     
&+\frac{1}{2}\int_0^{t}\int_Q \sum_{k=1}^\infty |\sigma^{ik}_s \D_iu_s+ \mu^ku_s+g^k_s|^2\phi_n''(u_s) dxds,\label{eq: ito before the limit}
\end{align}
for any  $t \in [0,T]$ (see for example, Section 3 in \cite{KIto}). 
By Young's inequality, and the parabolicity condition we have
for any $\varepsilon>0$, 
\begin{align}
\int_Q \phi_n(u_{t})dx  &\leq m^{(n)}_{ t}+\int_Q \phi_n(u_0)dx \nonumber \\
&+\int_0^{t} \int_Q (-\lambda |\nabla u_s|^2+ \varepsilon |\nabla u_s|^2+ N\sum_{i=1}^d|f^i_s|^2 )\phi''_n(u_s) dx ds \nonumber \\
 &+ \int_0^{t}\int_Q(\epsilon |\nabla u_s|^2+ N|u_s|^2+N \sum_{k=1}^\infty |g^k_s|^2)\phi_n''(u_s) dxds\nonumber\\                             
 &+\int_0^{ t}\int_Q ( b^i_s\D_i u_s  + c_su_s  +f^0_s)\phi'_n(u_s) dx ds ,\label{eq: estimate 1} 
\end{align}
where $N=N(d,K,\epsilon)$, and $m^{(n)}_t$ is the martingale from \eqref{eq: ito before the limit}.
One can check that the following inequalities hold,
\begin{enumerate}
\item[i)] $|r\phi'_n(r)| \leq p\phi_n(r)$
\item[ii)] $|r^2\phi''(r)| \leq p(p-1) \phi_n(r)$
\item[iii)] $|\phi_n'(r)|^2 \leq 4  p \ \phi_n''(r) \phi_n(r)$  
\item[iv)] $[\phi''_n(r)]^{p/(p-2)} \leq [p(p-1)]^{p/(p-2)} \phi_n(r)$,
\end{enumerate}
which combined with Young's inequality imply,

\begin{enumerate}
\item[i)]$\D_i u_s \phi'_n(u_s) \leq \epsilon\phi''_n(u_s)|\D_iu_s|^2+N\phi_n(u_s)$
\item[ii)] $|u_s \phi_n'(u_s)| \leq p\phi_n(u_s)$
\item[iii)] $|f^0_s\phi_n'(u_s)| \leq |f^0_s||\phi''_n(u_s)|^{1/2}|\phi_n(u_s)|^{1/2}
\leq N |f^0_s|^p+N \phi_n(u_s)
$
\item[iv)]   $|u_s|^2\phi_n''(u_s)\leq N \phi_n(u_s)$
\item[v)] $\sum_k |g_s^k|^2 \phi''_n(u_s) \leq N \phi_n(u_s)+ N \Big(\sum_k|g^k_s|^2\Big)^{p/2}$
\item[vi)]  $\sum_{i=1}^d|f^i_s|^2 \phi''_n(u_s)  \leq N \phi_n(u_s)+ N \sum_{i=1}^d |f^i_s|^p$,
\end{enumerate}
where  $N$ depends only on $p$ and $\epsilon$.

By choosing $\epsilon$ sufficiently small, and  taking expectations we obtain 
$$
\E \int_Q \phi_n(u_{ t})dx  +\E I_A \int_0^{ t} \int_Q |\nabla u_s|^2\phi''_n(u_s)dx ds \leq N\E \mathcal{K}_{t}+N\int_0^t \E\int_Q\phi_n(u_{ s}) dxds,
$$
where $N=N(d,p,K,\lambda)$ and
$$
\mathcal{K}_{t}=|\psi|^p_p+\int_0^t\sum_{l=0}^d |f^l_s|_p^p+|g_s|^p_p ds .
$$
By Gronwall's lemma we get 
$$
\E \int_Q \phi_n(u_{ t}) dx +\E \int_0^{t} \int_Q |\nabla u_s|^2\phi''_n(u_s)dx ds \leq N \E\mathcal{K}_{t}
$$
for any $t \in [0,T]$,  with $N=N(T,d,p,K,\lambda)$.
Going back to \eqref{eq: estimate 1}, using the same estimates, and the above relation, by taking suprema up to $T$ we have
$$
\E\sup_{t\leq T} \int_Q \phi_n(u_{  t})dx \leq N\E  I_A \mathcal{K}_{t} +\E \sup_{t\leq T} |m^{(n)}_t|.
$$
$$
\leq N\E \mathcal{K}_{T}+ N \E \left(\int_0^{T} \sum_k \left( \int_Q|\sigma^{ik} \D_iu_s+\mu^ku_s+g^k_s||\phi''_n(u_s)\phi_n(u_s)|^{1/2} dx\right)^2 ds \right)^{1/2}
$$
$$
\leq N\E \mathcal{K}_{T}+ N\E  \left(\int_0^{T} \int_Q (|\nabla u_s|^2+|u_s|^2+\sum_{k=1}^\infty |g^k_s|^2 )\phi''_n(u_s) dx\int_Q \phi_n(u_s)dx  ds
\right)^{1/2}
$$
$$
\leq N\E  \mathcal{K}_{T}+\frac{1}{2}\E\sup_{t \leq {T}} \int_Q \phi_n(u_t) dx< \infty,
$$
where $N=N(T,d,p,K,\lambda)$. Hence, 
$$
\E \sup_{t \leq T}\int_Q \phi_n(u_{t}) dx +\E  \int_0^{T} \int_Q |\nabla u_s|^2\phi''_n(u_s) dxds \leq N\E  \mathcal{K}_{T},
$$
and by Fatou's lemma we get \eqref{eq: finiteness of |u|^p}. For \eqref{eq:Ito formula for |u|^p}, we go back to \eqref{eq: ito before the limit}, and by letting a subsequence $n(k) \to \infty$ and using the dominated convergence theorem, we see that each term converges to the corresponding one in \eqref{eq:Ito formula for |u|^p} almost surely, for all $t\leq T$. This finishes the proof.

\end{proof}

\begin{corollary}\label{ito-corollary}

 Let $\gamma>1$ and denote $\kappa=4\gamma/(\gamma-1)$. Suppose furthermore that $r,r',q,q'\in(1,\infty)$, satisfying $1/r+2/r'=1$ and $1/q+2/q'=1$. Suppose that $u$ satisfies the conditions of Lemma \ref{lem: moment bounds for bounded f}  for any $p\in\{2\gamma^n,n\in\mathbb{N}\}$. Then, for any  $p\in\{2\gamma^n,n\in\mathbb{N}\}$, almost surely, for all $t\leq T$

$$
\int_Q|u_t|^pdx +\frac{p^2}{4}\int_0^t\int_Q|\nabla u_t|^2|u_t|^{p-2}dxds\leq N' m_t
$$
\begin{equation}\label{eq:corollary}
+ N\left[|\psi|_p^p+p^{\kappa}\|u\|_{r'p/2,q'p/2}^{p}+p^{-p}(\|f^0\|_{r,q}^p+\sum_{i=1}^d\|f^i\|_{2r,2q}^p+\||g|_{l_2}\|_{2r,2q}^p)\right],
\end{equation}
where $m_t$ is the martingale from \eqref{eq:Ito formula for |u|^p}, and $N,N'$ are constants depending only on $K,d,T,\lambda, |Q|,r,q$.

\end{corollary}

\begin{proof}

By Lemma \ref{lem: moment bounds for bounded f}, the parabolicity condition,  and Young's inequality we have

$$
\int_Q|u_t|^pdx  +\frac{p^2}{4}\int_0^t\int_Q|\nabla u_s|^2|u_s|^{p-2}dxds\leq N'm_t + N_1\left(\int_Q|\psi|^pdx\right.
$$
$$
+\int_0^t\left[\int_Q p^2|u_s|^p\left.+p|f^0_s||u_s|^{p-1}+p^2\sum_{i=1}^d|f^i_s|^2|u_s|^{p-2}+p^2 |g_s|_{l_2}^2|u_s|^{p-2}dx\right]ds\right)
.
$$
Then by H\"older's inequality  we have 
$$
\int_0^t\int_Q|f^0_s||u_s|^{p-1}dxds\leq\|f^0\|_{r,q}\|u\|_{q'(p-1)/2,r'(p-1)/2}^{p-1},
$$
and by Young's inequality we obtain
\begin{align*}
p\|f^0\|_{r,q}\|u\|_{q'(p-1)/2,r'(p-1)/2}^{p-1}&\leq p^{-p}\|f^0\|_{r,q}^p+p^{\kappa}\|u\|_{r'(p-1)/2,q'(p-1)/2}^{p}\\
&\leq p^{-p}\|f^0\|_{r,q}^p+N_2p^{\kappa}\|u\|_{r'p/2,q'p/2}^{p}.
\end{align*}
Similarly, for $n\geq 1$,
\begin{align*}
p^2\int_0^t\int_Q|f^i_s|^2|u_s|^{p-2}dxds &\leq p^2 \|f^i\|_{2r,2q}^2\|u\|_{r'(p-2)/2,q'(p-2)/2}^{p-2}\\
&\leq p^{-p}\|f^i\|_{2r,2q}^p+p^{\kappa}\|u\|_{r'(p-2)/2,q'(p-2)/2}^{p}\\
&\leq p^{-p}\|f^i\|_{2r,2q}^p+N_3p^{\kappa}\|u\|_{r'p/2,q'p/2}^{p}.
\end{align*}
The same holds for $g$ in place of $f^i$. The case $n=0$ can be covered separately with another constant $N_4$, and then $N$ can be chosen to be $\max\{N_1(N_2+N_3), N_4\}$. This finishes the proof.

\end{proof}

\begin{lemma}                 \label{lem: estimate for low f}
Suppose that $u$ satisfies equation  \eqref{SPDE},  $f^l \in \mathbb{L}_p$, for $l\in \{0,...,d\}$,  $\  g \in \mathbb{L}_p(l_2)$, and $\psi \in \mathbf{L}_p$ for some $p\geq2$. 
Then for any $0< \eta < p$, and for any $\epsilon>0$, 
$$                         
\E \left( \sup_{t \leq T} |u_t|^p_p + \frac{p^2}{4} \E\int_0^T \int_Q |\nabla u_s|^2|u_s|^{p-2} dxds\right)^{\eta/p} 
$$
$$
\leq  \epsilon \E\|u\|_{\infty}^\eta +N(\epsilon,p) \E \left[|\psi|_p^\eta +\|f^0\|_{1}^\eta+\sum_{i=1}^d\|f^i\|_{2}^\eta+\||g|_{l_2}\|_{2}^\eta \right]
$$
where $N(\epsilon,p)$  is a constant depending only on $\epsilon, \eta,  K,d,T,\lambda, |Q|$, and $p$.

\end{lemma}

\begin{proof}
As in the proof of corollary \ref{ito-corollary}, for any $\mathscr{F}_0-$measurable set $B$, we have almost surely
$$
I_B \int_Q|u_t|^pdx  +\frac{p^2}{4}I_B \int_0^t\int_Q|\nabla u_s|^2|u_s|^{p-2}dxds\leq N'I_Bm_t + N_1I_B \left(\int_Q|\psi|^pdx\right.
$$
\begin{equation}             \label{eq: after Ito}
+\int_0^t\left[\int_Q p^2|u_s|^p\left.+p|f^0_s||u_s|^{p-1}+p^2\sum_{i=1}^d|f^i_s|^2|u_s|^{p-2}+p^2 |g_s|_{l_2}^2|u_s|^{p-2}dx\right]ds\right),
\end{equation}
for any $t\in [0,T]$. The above relation, by virtue of Gronwal's lemma implies that for any stopping time $\tau \leq T$ 
\begin{equation}     \label{eq:estimate for the derivatives}
\sup_{t\leq T} \E I_B \int_Q|u_{t\wedge \tau}|^pdx  +\E I_B \int_0^\tau \int_Q|\nabla u_s|^2|u_s|^{p-2}dxds \leq N \E I_B \mathscr{V}_\tau, 
\end{equation}
where 
 $$
 \mathscr{V}_t:= \int_Q|\psi|^pdx
+\int_0^t \int_Q  |f^0_s||u_s|^{p-1}+\sum_{i=1}^d|f^i_s|^2|u_s|^{p-2}+ |g_s|_{l_2}^2|u_s|^{p-2}dx ds.
$$
Going back to \eqref{eq: after Ito}, and taking suprema up to $\tau$ and expectations, and having in mind \eqref{eq:estimate for the derivatives},  gives 
$$
\E \sup_{t \leq \tau}I_B \int_Q|u_t|^pdx  \leq N \E \sup_{t \leq \tau}  I_B|m_t| +N \E I_B \mathscr{V}_\tau.
$$
By the Burkholder-Gundy-Davis inequality  and \eqref{eq:estimate for the derivatives}  we have
 \begin{align}  \nonumber 
 \E \sup_{t \leq \tau}  I_B|m_t|  & \leq N \E I_B \left( \int_0^\tau
 \left( \int_Q |u_t|^{p-2}\left( |\nabla u_t|+|u_t|+|g|_{l_2}\right) dx  \right)^2 dt \right)^{1/2}  
 \\   \nonumber 
 & \leq N \E I_B  \left( \int_0^\tau
  \int_Q |u_t|^pdx \int_Q ( |\nabla u_t|^2 +|u_t|^2+|g|_{l_2}^2) |u|^{p-2} dx  dt \right)^{1/2}  
 \\ \nonumber 
 &\leq \frac{1}{2} \E \sup_{t\leq \tau } I_B \int_Q|u_t |^pdx 
 + N \E I_B \mathscr{V}_\tau.
\end{align}
Hence,
$$
\E \sup_{t \leq \tau}I_B \int_Q|u_t|^pdx  \leq N \E I_B \mathscr{V}_\tau,
$$
which combined with \eqref{eq:estimate for the derivatives}, by virtue of Lemma \ref{Revuz-Yor} gives 
$$
\E \left( \sup_{t \leq T} |u_t|^p_p + \frac{p^2}{4} \E\int_0^T \int_Q |\nabla u_s|^2|u_s|^{p-2} dxds\right)^{\eta/p}  \leq N \E \mathscr{V}_T^{\eta/p}
$$
\begin{align} \nonumber 
&\leq N \E \left[ |\psi|^p_p +\|u\|_\infty^{p-1}\|f^0\|_1+\|u\|_\infty^{p-2} \left( \sum_{i=1}^d \|f^i\|_2^2+\||g|_{l_2}\|^2_2\right) \right]^{\eta/p} \\
\nonumber 
&\leq  \epsilon \E\|u\|_{\infty}^\eta +N \E \left[|\psi|_p^\eta +\|f^0\|_{1}^\eta+\sum_{i=1}^d\|f^i\|_{2}^\eta+\||g|_{l_2}\|_{2}^\eta \right], 
\end{align}
which brings the proof to an end.
\end{proof}

\section{Proof of Theorem \ref{thm:main theorem}}\label{section-proof}

\begin{proof} Throughout the proof, the constants $N$ in our calculations will be allowed to depend on $\eta,r,q$ as well as on the structure constants. Notice that we may, and we will assume that $r,q < \infty$. Without loss of generality  we assume that the right hand side in $\eqref{eq:main estimate}$ is finite. Also, in the first part of the proof we make the assumption that $\psi$, $f^l$, $l=0,\ldots,d$, and $g$ are bounded by a constant $M$. in particular, by $\eqref{eq: finiteness of |u|^p}$, $u\in L_{\eta}(\Omega, L_{r,q})$ for any $\eta,r,q$.

Let us introduce the notation
$$
\mathcal{M}_{r,q,p}(t)=\|\mathbf{1}_{[0,t]}f^0\|_{r,q}^p+\sum_{i=1}^d\|\mathbf{1}_{[0,t]}f^i\|_{2r,2q}^p+\|\mathbf{1}_{[0,t]}|g|_{l_2}\|_{2r,2q}^p.
$$
Since $(r,q)\in \Gamma_d$, if we define $r'$ and $q'$ by $1/r+2/r'=1$, $1/q+2/q'=1$, we have
$$
\frac{d}{4}<\frac{1}{r'}+\frac{d}{2q'}=:\gamma\frac{d}{4}
$$
for some $\gamma>1$. Then $\hat{r}=\gamma r'$ and $\hat{q}=\gamma q'$ satisfy
$$
\frac{1}{\hat{r}}+\frac{d}{2\hat{q}}=\frac{d}{4}.
$$
By applying Lemma \ref{embedding} to $\hat{r},\hat{q}$, and $\bar{v}=|v|^{p/2}$, we have, for any $p\geq 2$
$$
\E \left[|\psi|_{\infty}^{\eta}\vee \left( \int_0^T\left(\int_Q |v_t|^{\hat{q}p/2}dx \right)^{\hat{r}/\hat{q}}dt \right)^{2\eta/\hat{r}p}\right]
$$
\begin{equation}\label{bound from below}
\leq \left[\E |\psi|_{\infty}^{\eta}\vee  N^{\eta/p}\left(\sup_{0\leq t\leq T}\int_Q|v_t|^{p}dx+\frac{p^2}{4}\int_0^T\int_Q|\nabla v_t|^2|v_t|^{p-2}dx dt\right)^{\eta/p}\right].
\end{equation}
To estimate the right-hand side above, first notice that, if $p=2\gamma^n$ for some $n$, then  by taking supremum in \eqref{eq:corollary}, we have for any stopping time $\tau\leq T$, and any $\mathscr{F}_0-$ measurable set $B$,
$$
I_B \sup_{0\leq s\leq\tau}\int_Q|v_s|^{p}dx
$$
\begin{equation}\label{intermediate}
\leq N  I_B \left(|\psi|_{\infty}^{p}+p^{\kappa}\|\mathbf{1}_{[0,\tau]}v\|_{r'p/2,q'p/2}^{p}+p^{-p}\mathcal{M}_{r,q,p}(\tau)\right)
+N'  I_B \sup_{0\leq s \leq \tau}|m_s|,
\end{equation}
By the Davis inequality we can write
$$
\E  I_B \sup_{0\leq s \leq \tau}|m_s|\leq N\E  I_B \left(\int_0^{\tau}\sum_k \left(\int_Q p(\sigma^{ik}_s \D_i v_s +\mu^kv_s+g^k) v_s|v_s|^{p-2}dx \right)^2ds\right)^{1/2}
$$
$$
\leq N\E  I_B \left(\sup_{0\leq s\leq\tau}\int_Q|v_s|^pdx \right)^{1/2}\left(\int_0^{\tau}\int_Qp^2\sum_k|\sigma^{ik}_s \D_i v_s +\mu^kv_s+g^k|^2|v_s|^{p-2}dxds\right)^{1/2}.
$$
Applying Young's inequality and recalling the already seen estimates in the proof of Corollary \ref{ito-corollary} (i) for the second term yields
$$
\E  I_B \sup_{0\leq s \leq \tau}|m_s|\leq \varepsilon\E  I_B \sup_{0\leq s\leq\tau}\int_Q|v_s|^p dx
$$
$$
+\frac{N}{\varepsilon}\E  I_B \left(p^2\int_{0}^{\tau}\int_Q|\nabla v_s|^2|v_s|^{p-2}dx ds+p^{\kappa}\|\mathbf{1}_{[0,\tau]}v\|_{r'p/2,q'p/2}^{p}+p^{-p}\|\mathbf{1}_{[0,\tau]}|g|_{l_2}\|_{2r,2q}^p\right)
$$
for any $\varepsilon>0.$ With the appropriate choice of $\varepsilon,$ combining this with (\ref{intermediate}) and using \eqref{eq:corollary} once again, now without taking supremum, we get

$$
\E I_B \left(\sup_{0\leq s\leq \tau}\int_Q|v_s|^p dx+\frac{p^2}{4}\int_0^{\tau}\int_Q|\nabla v_s|^2|v_s|^{p-2} dx ds\right) 
$$
\begin{align}
\nonumber
& \leq N  \E  I_B\Big( |\psi|_{\infty}^p+ p^2 \int_{0}^{\tau}\int_Q|\nabla v_s|^2|v_s|^{p-2}dx ds
+p^{\kappa}\|\mathbf{1}_{[0,\tau]}v\|_{r'p/2,q'p/2}^{p}+p^{-p}\mathcal{M}_{r,q,p}(\tau)\Big)\\ \nonumber 
& \leq N\E   I_B \left(|\psi|_{\infty}^p+p^{\kappa}\|\mathbf{1}_{[0,\tau]}v\|_{r'p/2,q'p/2}^{p}+p^{-p}\mathcal{M}_{r,q,p}(\tau)\right)+N'\E  I_B m_{\tau},
\end{align}
and the last expectation vanishes. Now consider
$$
X_t=|\psi|_{\infty}^p\vee\left( \sup_{0\leq s\leq t}\int_Q|v_s|^p dx+\frac{p^2}{4}\int_0^t\int_Q|\nabla v_s|^2|v_s|^{p-2} dx ds \right)
$$
and
$$
A_t=Cp^{\kappa}\left(|\psi|_{\infty}^p\vee \|\mathbf{1}_{[0,t]}v\|_{r'p/2,q'p/2}^{p}+p^{-p}\mathcal{M}_{r,q,p}(t)\right)
$$
for a large enough, but fixed $C$. The argument above gives that
\begin{align}
\nonumber 
\E I_B X_{\tau} & \leq \E I_B \left(|\psi|_{\infty}^p+\sup_{0\leq s\leq \tau}\int_Q|v_s|^pdx +\frac{p^2}{4}\int_0^{\tau}\int_Q|\nabla v_s|^2|v_s|^{p-2}dx ds\right)
\\ \nonumber 
& \leq N\E I_B \left(|\psi|_{\infty}^p+p^{\kappa}\|\mathbf{1}_{[0,\tau]}v\|_{r'p/2,q'p/2}^{p}+p^{-p}\mathcal{M}_{r,q,p}(\tau)\right)\leq \E I_B A_{\tau}.
\end{align}
Therefore the condition of Proposition \ref{Revuz-Yor} is satisfied, and thus for $\eta< p$ we obtain
$$
\E \left(|\psi|_{\infty}^p\vee \left(\sup_{0\leq t\leq T}\int_Q|v_t|^pdx +\frac{p^2}{4}\int_0^T \int_Q|\nabla v_t|^2|v_t|^{p-2}dx dt\right) \right)^{\eta/p}
$$
\begin{align}  \nonumber 
& \leq (Np^{\kappa+1})^{\eta/p}\frac{p}{p-\eta}\E\left(|\psi|_{\infty}^p\vee \|v\|_{r'p/2,q'p/2}^{p}+p^{-p}\mathcal{M}_{r,q,p}(T)\right)^{\eta/p} \\ 
\label{bound from above}
& \leq (Np^{\kappa+1})^{\eta/p}\frac{p}{p-\eta}\E\left(|\psi|_{\infty}^{\eta}\vee \|v\|_{r'p/2,q'p/2}^{\eta}+p^{-\eta}\mathcal{M}_{r,q,\eta}(T)\right).
\end{align}
Let us choose $p=p_n=2\gamma^n$ for $n\geq0$, and use the notation $c_n=(Np_n^{\kappa+1})^{\eta/p_n}\frac{p_n}{p_n-\eta}$. Upon combining \eqref{bound from below} and \eqref{bound from above}, for $p_n>\eta$ we can write the following inequality, reminiscent of Moser's iteration:
\begin{equation}\label{iteratable}
\E |\psi|_{\infty}^{\eta}\vee\|v\|_{r'p_{n+1}/2,q'p_{n+1}/2}^{\eta}\leq c_n\E\left[|\psi|_{\infty}^{\eta}\vee\|v\|_{r'p_n/2,q'p_n/2}^{\eta}
+Np_n^{-\eta}\mathcal{M}_{r,q,\eta}(T)\right].
\end{equation}
Consider the minimal $n_0=n_0(d,\eta)$ such that $p_{n_0}>2\eta$. Taking any integer $m\geq n_0$ we have
\begin{align}
\nonumber 
\prod_{n=n_0}^{m}c_n & \leq\prod_{n=n_0}^{m}(N\gamma^{\kappa+1})^{\eta n/{2\gamma^n}}e^{2\eta/2\gamma^n} \\ \nonumber 
&=\exp\left[\log(N\gamma^{\kappa+1})\sum_{n=n_0}^{m}\frac{\eta n}{2\gamma^n}+\sum_{n=n_0}^{m}\frac{\eta}{\gamma^n}\right]\leq N_0,
\end{align}
where $N_0$ does not depend on $m$. Also,
$$
N\sum_{n=n_0}^{m}p_n^{-\eta}\leq N_1,
$$
where $N_1$ does not depend on $m$. Therefore, by iterating (\ref{iteratable}) we get
\begin{align}
\nonumber 
\liminf_{m\rightarrow\infty}\E |\psi|_{\infty}^{\eta}\vee \|v\|_{r'p_{m}/2,q'p_{m}/2}^{\eta}\leq & N_0N_1 \E \mathcal{M}_{r,q,\eta}(T) \\
\nonumber &+N_0 \E |\psi|_{\infty}^{\eta}\vee \|v\|_{r'({p_{n_0+1}})/2,q'({p_{n_0+1}})/2}^{\eta}, 
\end{align}
and thus by Fatou's lemma
\begin{equation}\label{uniform-estimate1}
\E\| v\|_{\infty}^{\eta}\leq N\E(|\psi|_{\infty}^{\eta} \vee\|v\|^\eta_{r'(p_{n_0+1})/2,q'(p_{n_0+1})/2}+ \mathcal{M}_{r,q,\eta}(T)),
\end{equation} 
in particular, the left-hand side is finite. By Lemma \ref{lem: estimate for low f} we get 
$$
\E \left(|\psi|_{\infty}^p\vee\left( \sup_{0\leq t\leq T}\int_Q|v_t|^pdx +\frac{p^2}{4}\int_0^T \int_Q|\nabla v_t|^2|v_t|^{p-2}dx dt\right) \right)^{\eta/p}
$$
\begin{equation}\label{whatever}
\leq \epsilon\E\|v\|_{\infty}^{\eta}+N(\epsilon,p)\E\left(|\psi|_{\infty}^{\eta}+\mathcal{M}_{1,1,\eta}(T)\right)
\end{equation}
for any $\epsilon>0$. Combining \eqref{bound from below} and \eqref{whatever} for $p=p_{n_0}$ gives
$$
\E |\psi|_{\infty}^{\eta}\vee\|v\|_{r'(p_{n_0+1})/2,q'(p_{n_0+1})/2}^{\eta}=\E |\psi|_{\infty}^{\eta}\vee\|v\|_{\hat{r}p_{n_0}/2,q'p_{n_0}/2}^{\eta}
$$
\begin{equation}\label{whatever2}
\leq \epsilon\E\|v\|_{\infty}^{\eta}+N(\epsilon,p_{n_0})\E\left(|\psi|_{\infty}^{\eta}+\mathcal{M}_{1,1,\eta}(T)\right).
\end{equation}
Choosing $\epsilon$ sufficiently small, plugging \eqref{whatever2} into \eqref{uniform-estimate1}, and rearranging yields the desired inequality
\begin{equation}\label{uniform-estimate2}
\E\| v\|_{\infty}^{\eta}\leq N\E(|\psi|_{\infty}^{\eta}+\mathcal{M}_{r,q,\eta}(T)).
\end{equation}

As for the general case, set
$$\psi^{(n)}=\psi\wedge n,\;\;\;f^{l,(n)}=f^l\wedge n,\;\;\;g^{k,(n)}=g^k\wedge (n/k),$$
define $\mathcal{M}_{r,q,p}^{(n)}$ correspondingly, and let $v^n$ be the solution of the corresponding equation. This new data is now bounded by a constant, so the previous argument applies, and thus
$$\E\| v^n\|_{\infty}^{\eta}\leq N\E(|\psi^{(n)}|_{\infty}^{\eta}+\mathcal{M}_{r,q,\eta }^{(n)}(T)\leq
 N\E(|\psi|_{\infty}^{\eta}+\mathcal{M}_{r,q,\eta}(T)).$$
Since $v^n\rightarrow v$ in $\mathbb{L}_2$, for a subsequence $k(n)$, $v^{k(n)}\rightarrow v$ for almost every $\omega,t,x$. In particular, almost surely $\|v\|_{\infty}\leq\liminf_{n\rightarrow\infty}\|v^{k(n)}\|_{\infty}$, and by Fatou's lemma
$$\E\| v\|_{\infty}^{\eta}\leq\liminf_{n\rightarrow\infty}\E\|v^{k(n)}\|_{\infty}^{\eta}\leq N\E(|\psi|_{\infty}^{\eta}+\mathcal{M}_{r,q,\eta}(T)).$$

\end{proof}

\section{Semilinear SPDEs}             \label{section-semilinear}
In this section, we will use the uniform norm estimates obtained in the previous section, to construct solutions for the following equation
\begin{equation}                                                  \label{SLSPDE} 
du_t=( L_t u_t +f_t(u_t))dt
+(M^k_t u_t+g^k_t)dw^k_t, \\
  u_0=\psi
\end{equation}
for $(t,x) \in [0,T] \times Q$, where  $f$ is a real function defined on $\Omega \times[0,T] \times Q \times \bR$ and is 
$\mathscr{P} \times \mathscr{B}(\bR^d)\times \mathscr{B}(\bR)- $measurable.

\begin{assumption}                              \label{Assumption_on_f}
 
The function $f$ satisfies the following 

i) for all $r,r' \in \bR$ and for all $(\omega,t,x)$  we have
$$
(r-r')(f_t(x,r)-f_t(x,r')) \leq K|r-r'|^2
$$

ii) For all $(\omega,t,x)$, $f_t(x,r)$ is continuous in $r$

iii) for all $N>0$, there exists a  function $h^N\in \mathbb{L}_2$ with $\E\|h^N\|_\infty < \infty$,
such that for any $(\omega,t,x)$
$$ 
|f_t(x,r)| \leq  |h^N_t(x)|,
$$
whenever $|r| \leq N$.

iv) $\E|\psi|_\infty+\E \||g|_{l_2}\|_\infty < \infty$
\end{assumption}
\begin{definition}                           \label{definition_of_solution} 
A solution of equation \eqref{SLSPDE} is an $\ \mathscr{F}_t-$adapted, strongly continuous process $(u_t)_{t\in [0,T]}$ with values in $L_2(Q)$ such that 

i) $u_t \in H^1_0, \ \text{for $dP \times dt$ almost every} \  (\omega,t) \  \in \Omega\times [0,T]$ 

ii) $  \int_0^T |u_t|_2^2+|\nabla u_t|_2^2  dt< \infty$  (a.s.)

iii) almost surely, $u$ is essentially bounded in $(t,x)$

iv) for all $\phi\in C_c^\infty(Q)$ we have with probability one
\[(u_t,\phi)=(\psi,\phi)+\int_0^t -(a^{ij}_s\D_iu_s,\D_j \phi)+(b^i_s\D_iu_s+c_su_s, \phi)+(f_s(u_s),\phi)ds \]
\[+\int_0^t(M^k_su_s+g^k_s,\phi)dw^k_s, \]
 for all $t \in [0,T]$.
 
\end{definition}  
Notice that by Assumption \ref{Assumption_on_f} iii), and (iii) from Definition \ref{definition_of_solution}, the term $\int_0^t(f_s(u_s),\phi)ds$ is meaningful.

\begin{theorem}                                              \label{thm: semilinear}
Under Assumptions \ref{as:coef and free terms}, \ref{as: parabolicity}, and \ref{Assumption_on_f}, there exists a unique solution of equation (\ref{SLSPDE}).
\end{theorem}
\begin{remark}
From now on we can and we will assume that the function $f$ is decreasing in $r$ or else, by virtue of Assumption \ref{Assumption_on_f},  we can replace $f_t(x,r)$ by 
$\tilde{f}_t(x,r):= f_t(x,r)-Kr$ and $c_t(x)$ with $\tilde{c}_t(x):= c_t(x)+K$.
\end{remark}
We will need the following particular case from \cite{KI}. We consider two equations 
\begin{equation}                                \label{eq:SPDE linear growth}   
du^i_t= (L_tu^i_t+f_t^i(u^i_t))dt + (M^k_tu^i_t+g^k_t)dw^k_t, \ u^i_0= \psi^i,
\end{equation}
for $i=1,2$. 
\begin{assumption}                                             \label{as: linear growth}
The functions $f^i$, $i=1,2$, are appropriately measurable, and there exists $h \in \mathbb{L}_2$ and a constant $C>0$, such that for any $\omega, t,x$, and for any $r\in \mathbb{R}$ we have
$$
|f^1_t(x,r)|^2+|f^2_t(x,r)|^2 \leq C|r|^2+|h_t(x)|^2.
$$
\end{assumption}
\begin{theorem}                            \label{thm: comparison}
Suppose that Assumptions \ref{as: parabolicity}, \ref{as:coef and free terms} and \ref{as: linear growth} hold. Let $u^i$, $i=1,2$ be the $L_2-$ solutions of the equations in \eqref{eq:SPDE linear growth}, for $i=1,2$ respectively. Suppose that $f^1 \leq f^2$, $\psi^1 \leq \psi^2$ and assume that either $f^1$ or $f^2$ satisfy Assumption \ref{Assumption_on_f}. Then, almost surely and for any $t \in [0,T]$, 
$u^1_t \leq u^2_t$ for almost every $x \in Q$.
\end{theorem}
\begin{proof}[Proof of Theorem \ref{thm: semilinear}]
We truncate the function $f$ by setting
 \begin{equation}
 f^{n,m}_t(x,r)= \left\{
\begin{array}{rl}
f_t(x,m) & \text{if } r> m \\
f_t(x,r) & \text{if } -n \leq r \leq m \\
f_t(x,-n) & \text{if } r < -n ,
\end{array} \right. \nonumber
\end{equation}
for $n, \ m \in\mathbb{N}$ we consider the equation
\begin{align}                           \label{trancated_spde}
\nonumber
du^{n,m}_t&=( L_t u^{n,m}_t+f^{n,m}_t(u^{n,m}_t))dt
+(M^k_t u^{n,m}_t+g^k_t) dw^k_t, \\
  u^{n,m}_0&=\psi
\end{align}
We first fix  $m \in \mathbb{N}$. Equation \eqref{trancated_spde} can be realised as a stochastic evolution equation on the triple $H^1_0 \hookrightarrow L_2(\bR^d) \hookrightarrow H^{-1}$. One can easily check that under Assumptions \ref{as:coef and free terms}, \ref{as: parabolicity} and \ref{Assumption_on_f},  the conditions $(A_1)$ through $(A_5)$  from Section 3.2 in \cite{KSEE}  are satisfied, and therefore  equation \eqref{trancated_spde} has a unique $L_2-$solution $(u^{n,m}_t)_{t \in [0,T]}$. We also have that for $n' \geq n$, $f^{n',m} \geq  f^{n,m}$. By Theorem \ref{thm: comparison} we get that  almost surely, for all $t\in[0,T]$
\begin{equation}                         \label{eq: com n,n'}
u^{n',m}_t(x) \geq u^{n,m}_t(x), \ \text{ for almost every  $x$}.
\end{equation}

We define now the stopping time
$$
\tau^{R,m}:= \inf \{t \geq 0 : \int_Q (u^{1,m}_t  +R)_-^2  dx>0\}\wedge T.
$$
We claim that for each $R\in \mathbb{N}$, there exists a set $\Omega_R$ of full probability, such that   for each $\omega\in \Omega_R$, and for all $n  \geq R$ we have that 
\begin{equation}                                    \label{um=uinfty}
u^{n,m}_t = u^{R,m}_t, \ \text{ for $t \in  [0, \tau^{R,m}]$}.
\end{equation}
Notice that by \eqref{eq: com n,n'} and the definition of $\tau^{R,m}$,  for all $n \geq R$ 
$$
f^{n,m}_t(x,u^{n,m}_t(x))=f^{R,m}_t(x,u^{n,m}_t(x)), \ \text{ for $t \in  [0, \tau^{R,m}]$} .
$$
This means that for all $n \geq R$ the processes $u ^{n,m}_t $ satisfies
\begin{align}                           \label{trancated_spde_stopped}
\nonumber
dv_t &= (L_t v_t+f^{R,m}_t(v_t))dt 
+(M^k_t v_t+g^k_t) dw^k_t, \\
  v_0&=\psi,
\end{align}
on $[0, \tau^{R,m}]$.
The uniqueness of the $L_2-$solution of the above equation shows \eqref{um=uinfty}.  Notice that by Assumption \ref{Assumption_on_f} (iii) and (iv),   Theorem \ref{thm:main theorem}  guarantees that  $u^{1,m}$ is almost surely essentially bounded in $(t,x)$. Therefore,  for almost every $\omega \in \Omega$, $\tau^{R,m} =T$ for all $R$ large enough. On the set $\tilde{\Omega}:=\cap_{R\in \mathbb{N}}\Omega_R$ we define $u^{\infty,m}_t= \lim_{n\to \infty} u^{n,m}_t$, where the limit is in the sense of $L_2(Q)$. Since for each $\omega\in \tilde{\Omega}$, we have  $u^{\infty,m}_t=u^{n,m}_t$ for all $t\leq \tau^{R,m}$, and for any $n \geq R$, it follows that the process $(u^{\infty,m}_t)_{t \in [0,T]}$ is an adapted continuous $L_2(Q)-$valued process such that

i) $u^{\infty,m}_t \in H^1_0, \ \text{for $dP \times dt$ almost every} \  (\omega,t) \  \in \Omega\times [0,T]$ 

ii) $ \int_0^T  |u^{\infty,m}_t|^2_2+|\nabla u^{\infty,m}_t|^2_2  dt < \infty $(a.s.)

iii) $u^{\infty,m}_t$ is almost surely essentially bounded in $(t,x)$

iv) for all $\phi\in C_c^\infty(Q)$ we have with probability one
\[(u^{\infty,m}_t,\phi)=\int_0^t (a^{ij}_s\D_{ij}u^{\infty,m}_s, \phi)+(b^i_s\D_iu^{\infty,m}_s+c_su^m_s, \phi)+(f^m_s(u^{\infty,m}_s),\phi)ds \]
\[+\int_0^t(\sigma^{ik}_s\D_iu^{\infty,m}_s+\nu^k_su^{\infty,m}_s+g^k_s,\phi)dw^k_s+(\psi,\phi), \]
 for all $t \in [0,T]$,
 where 
  \begin{equation}
 f^m_t(x,r)= \left\{
\begin{array}{rl}
f_t(x,m) & \text{if } r> m \\
f_t(x,r) & \text{if }   r \leq m. 
\end{array} \right. \nonumber
\end{equation}
Now we will let $m \to \infty$.  Let us define the stopping time
$$
\tau^R:= \inf \{ t \geq 0 : \int_Q ( u^ {\infty,1}_t-R)_+^2 dx >0\} \wedge T.
$$
As before we claim that for any $R>0$, there exists a set $\Omega'_R$ of full probability, such that for any $\omega \in \Omega'_R$ and any  $m, m' \geq R$,
\begin{equation}                           \label{eq:m,m'}
u^{\infty,m'} _t= u^{\infty, m}_t \ \text{ on $[0, \tau^R]$}.
\end{equation}
To show this it suffices to show  that for each $R\in \mathbb{N}$, almost surely, for all  $m \geq R$, we have
$u^{n, m}_t= u^{n, R}_t \ \text{ on $[0, \tau^R]$}$ for all $n \in \mathbb{N}$. To show this we set 
$$
\tau^R_n :=\inf \{ t \geq 0 : \int_Q  (u^ {n,1}_t-R)^2_+dx>0 \} \wedge T.
$$
For all $m \geq R$ we have that the processes $u^{n,m}_t$ satisfy the equation
\begin{align}                           
\nonumber
dv_t &=(L_t v_t +f^{n,R}_t(v_t))dt +(M^k_t v_t+g^k_t\} dw^k_t, \\
  v_0(x)&=\psi(x),
\end{align}
for $t \leq \tau^R_n$.
It follows  that almost surely, 
$u^{n, m}_t= u^{n, R}_t$ for $t \leq \tau^R_n$, for all $n$. We just note here that by the comparison principle again, we have $\tau^R \leq \tau^R_n$ and this shows \eqref{eq:m,m'}. Also for almost every  $\omega \in \Omega$, we have $\tau^R =T$ for $R$ large enough. Hence we can define $u_t= \lim_{m \to \infty}u^{\infty,m}_t$, and then one can easily see that $u_t$ has the desired properties.

For the uniqueness, let $u^{(1)}$ and $u^{(2)}$ be solutions of \eqref{SLSPDE}. Then one can define the stopping time 
$$
\tau_N= \inf\{t \geq0 : \int_Q(|u^{(1)}_t|-N)^2_+ dx\vee \int_Q(|u^{(2)}_t|-N)^2_+dx>0 \},
$$
 to see that for $t \leq \tau_N$, the two solutions satisfy equation \eqref{trancated_spde} with $n=m=N$, and the claim follows, since $\tau_N=T$ almost surely, for large enough $N$.
  
\end{proof}

\section*{Acknowledgements}
The authors would like to express their gratitude towards their PhD advisor, Professor Istv\'an Gy\"ongy, for his help and support during the preparation of this paper. They are furthermore grateful to the anonymous referee for her/ his helpful suggestions.



\end{document}